\documentclass[11pt,oneside,reqno]{article}

\usepackage[a4paper,width=15cm,height=24cm]{geometry}

\usepackage{amsmath,amsthm,amssymb}
\usepackage{bbm}
\usepackage{mathrsfs}

\usepackage[authoryear,longnamesfirst]{natbib}
\usepackage{array}

\usepackage{filemod}
\usepackage[breaklinks=true,hidelinks]{hyperref}

\theoremstyle{plain}
\newtheorem{theorem}{Theorem}[section]

\newtheorem{lemma}[theorem]{Lemma}
\newtheorem{corollary}[theorem]{Corollary}

\theoremstyle{definition}

\newtheorem{remark}[theorem]{Remark}

\makeatletter

\def\bes#1{\begin{equation*}\begin{split}#1\end{split}\end{equation*}}
\def\besn#1{\begin{equation}\begin{split}#1\end{split}\end{equation}}

\def\given{\typeout{Command 'given' should only be used within bracket command}}
\newcounter{@bracketlevel}
\def\@bracketfactory#1#2#3#4#5#6{
\expandafter\def\csname#1\endcsname##1{%
\addtocounter{@bracketlevel}{1}%
\global\expandafter\let\csname @middummy\alph{@bracketlevel}\endcsname\given%
\global\def\given{\mskip#5\csname#4\endcsname\vert\mskip#6}\csname#4l\endcsname#2##1\csname#4r\endcsname#3%
\global\expandafter\let\expandafter\given\csname @middummy\alph{@bracketlevel}\endcsname
\addtocounter{@bracketlevel}{-1}}%
}
\def\bracketfactory#1#2#3{%
\@bracketfactory{#1}{#2}{#3}{relax}{1mu plus 0.25mu minus 0.25mu}{0.6mu plus 0.15mu minus 0.15mu}
\@bracketfactory{b#1}{#2}{#3}{big}{1mu plus 0.25mu minus 0.25mu}{0.6mu plus 0.15mu minus 0.15mu}
\@bracketfactory{bb#1}{#2}{#3}{Big}{2.4mu plus 0.8mu minus 0.8mu}{1.8mu plus 0.6mu minus 0.6mu}
\@bracketfactory{bbb#1}{#2}{#3}{bigg}{3.2mu plus 1mu minus 1mu}{2.4mu plus 0.75mu minus 0.75mu}
\@bracketfactory{bbbb#1}{#2}{#3}{Bigg}{4mu plus 1mu minus 1mu}{3mu plus 0.75mu minus 0.75mu}
}
\bracketfactory{clc}{\lbrace}{\rbrace}
\bracketfactory{clr}{(}{)}
\bracketfactory{cls}{[}{]}
\bracketfactory{abs}{\lvert}{\rvert}
\bracketfactory{norm}{\Vert}{\Vert}
\bracketfactory{floor}{\lfloor}{\rfloor}
\bracketfactory{ceil}{\lceil}{\rceil}
\bracketfactory{angle}{\langle}{\rangle}

\newcounter{ctr}\loop\stepcounter{ctr}\edef\X{\@Alph\c@ctr}%
	\expandafter\edef\csname s\X\endcsname{\noexpand\mathscr{\X}}
	\expandafter\edef\csname c\X\endcsname{\noexpand\mathcal{\X}}
	\expandafter\edef\csname b\X\endcsname{\noexpand\boldsymbol{\X}}
	\expandafter\edef\csname I\X\endcsname{\noexpand\mathbbm{\X}}
	\expandafter\edef\csname r\X\endcsname{\noexpand\mathrm{\X}}
\ifnum\thectr<26\repeat

\newcount\minute
\newcount\hour
\newcount\hourMins
\def\now{%
\minute=\time%
\hour=\time \divide \hour by 60%
\hourMins=\hour \multiply\hourMins by 60%
\advance\minute by -\hourMins%
\zeroPadTwo{\the\hour}:\zeroPadTwo{\the\minute}%
}
\def\zeroPadTwo#1{\ifnum #1<10 0\fi#1}

\numberwithin{equation}{section}
\allowdisplaybreaks[4]

\renewcommand\section{\@startsection {section}{1}{\z@}%
{-3.5ex \@plus -1ex \@minus -.2ex}%
{1.3ex \@plus.2ex}%
{\center\small\sc\mathversion{bold}\MakeUppercase}}

\def\subsection#1{\@startsection {subsection}{2}{0pt}%
{-3.5ex \@plus -1ex \@minus -.2ex}%
{1ex \@plus.2ex}%
{\bf\mathversion{bold}}{#1}}

\def\subsubsection#1{\@startsection{subsubsection}{3}{0pt}%
{\medskipamount}%
{-10pt}%
{\normalsize\itshape}{\kern-2.2ex. #1.}}

\def\blfootnote{\xdef\@thefnmark{}\@footnotetext}

\makeatother

\renewcommand{\cite}{\citet}

\def\^#1{\ifmmode {\mathaccent"705E #1} \else {\accent94 #1} \fi}
\def\~#1{\ifmmode {\mathaccent"707E #1} \else {\accent"7E #1} \fi}

\edef\-#1{\noexpand\ifmmode {\noexpand\bar{#1}} \noexpand\else \-#1\noexpand\fi}
\def\>#1{\vec{#1}}
\def\.#1{\dot{#1}}

\def\atop{\@@atop}

\renewcommand{\leq}{\leqslant}
\renewcommand{\geq}{\geqslant}
\renewcommand{\phi}{\varphi}

\newcommand{\eq}{\eqref}

\begin{document}

\title{\sc\bf\large\MakeUppercase{Increasing Gambler's Ruin duration and Brownian Motion exit times}}
\author{\sc Steven Evans, Erol A. Pek\"oz, and Rhonda Righter}
\date{\it University of California, Berkeley Department of Statistics, Boston University Questrom School of Business, and University of California, Berkeley Department of Industrial Engineering and Operations Research}

\maketitle

\begin{abstract}
In Gambler's Ruin when both players start with the same amount of money, we show the playing time stochastically increases when the games are made more fair.  We give two different arguments for this fact that extend results from \cite{Pek2021}. We then use this to show that the exit time from a symmetric interval for Brownian motion with drift  stochastically increases as the drift moves closer to zero; this result is not easily obtainable from available explicit formulas for the density. \end{abstract}
%
%
\section{Introduction}
For a simple random walk $S_n = \sum_{i=1}^n X_i,$ where each $X_i$ is $+1$ with probability $p$ and $-1$ with probability $1-p$, we are interested in $T^{|k|},$ the time until the random walk hits either $k$ or $-k,$ where $k$ is a specified positive integer. It should be noted that $T^{|k|}$ is the duration of the Gambler's Ruin when both players start with $k.$  It was recently shown  in \cite{Pek2021} that $T^{|k|}$ for any $k$ is stochastically maximized when $p=1/2$, and previously in \cite{Zhang2021} where the weaker result in expectation was shown.  In this paper we prove the stronger result that $T^{|k|}$ stochastically increases as $p$ gets closer to 1/2 using two very different arguments.  We then use this to show that the exit time from a symmetric interval for Brownian motion with drift stochastically increases when the drift moves closer to zero, which appears also to be a new result and not easily obtainable from available formulas for the exit time probability density function.  
We also show that, for all $p$ and $k$, $T^{|k|}$ is independent of which player wins the game.
It should be noted that the technical challenges involved have left a scarcity of results for Gambler's Ruin with more than two players,  see \cite{Diaconis2021} and the references therein.

The Gambler's Ruin problem is one of the oldest problems in probability. 
As told by \cite{Song2013},  computing the chances each player wins was solved 
by Pascal and Fermat and appeared with four other problems in the earliest book   published on probability theory (\cite{Huygens1657}).
Study of the distribution of the duration of the game started with \cite{deMoivre1711} and was taken up in modern times with
\cite{Feller1968}{[Chapter 14.5]}, who derived the formula
\bes{
\IP(T^{|k|}=n) = & k^{-1} 2^{n+1}[p^{(n+k)/2} (1-p)^{(n-k)/2}+p^{(n-k)/2} (1-p)^{(n+k)/2}]\\ & \times \sum_{j=1}^{k-1}
\cos^{n-1}\left(\frac{\pi j}{2k}\right)
\sin\left(\frac{\pi j}{2k}\right)
\sin\left(\frac{\pi j}{2}\right)
}
for point probabilities when $n-k$ is even, and
\cite{Karni1977}, who derived
\bes{\IP(& T^{|k|}=n)  \\ & =  \left[ {n-1\choose \frac{1}{2}(n-k)} - {n-1\choose \frac{1}{2}(n-3k) }  
- {n-1\choose \frac{1}{2}(n+k) }
+{n-1\choose \frac{1}{2}(n-5k) } 
+{n-1\choose \frac{1}{2}(n+3k) }
 \right] \\ &\ \ \ \ \times [p^{(n+k)/2} (1-p)^{(n-k)/2}+p^{(n-k)/2} (1-p)^{(n+k)/2}]
 }
 in the case where $n\geq 5k$ and gave adjustments for $k\leq n <5k.$
 
Neither of the expressions above, however, seems  amenable to showing that the tail probabilities $\IP(T^{|k|}>n)$ increase (i.e., that $T^{|k|}$ stochastically increases) as $p$ moves towards $1/2$. The expressions can be easily seen to be increasing as $p$ moves towards 1/2 for sufficiently large $n$ (the derivative with respect to $p$ of either expression has the same sign as 
$n(1-2p)(p^k+(1-p)^k)+k(p^k-(1-p)^k)$), and hence the corresponding tail probabilities increase too.  The expressions for the point probabilities above can, however, decrease as $p$ moves towards 1/2 when $n$ is small, and thus not much can be easily said about tail probabilities in general.  
%
This leaves stochastic inequalities difficult to obtain from these expressions.
Our main result, which we state next, bypasses the use of these expressions and shows that moving $p$ towards 1/2 stochastically increases the game duration. 
\begin{theorem} \label{thm1}
$\IP(T^{|k|}>n)$ is increasing in $p$ for $0\leq p \leq 1/2$ and decreasing in $p$ for $1/2 \leq p \leq 1$.
\end{theorem}
We prove this using two different approaches in the next sections.
\section{A decomposition argument}
Let 
$T_0$ be the first time (after time 0) the walk revisits state 0.
As was done in \cite{Pek2021}, we can use the decomposition
$$T^{|k|}=Z+\sum_{i=1}^{N-1} Y_i$$
where $N\sim \mbox{Geometric} ( \IP(T_0 > T^{|k|}))$, $N-1$ is distributed as the number of times the walk returns to 0 before it first hits $\pm k$, $Z\sim (T^{|k|}|T_0 > T^{|k|})$ is the time to go from 0 to $\pm k$ given the walk does not return to 0, and $Y_i\sim (T_0|T_0<T^{|k|})$, $i=1,2,...$ are i.i.d. random variables representing the times to return to 0 given the walk returns to 0 before hitting $\pm k$. All of these
component random variables are independent.
We will show the following lemma, from which Theorem \ref{thm1} follows immediately.
\begin{lemma}\label{lem1}
$Y_i$, $Z$, and $N$ are all stochastically increasing (decreasing) in $p$ for $0\leq p \leq 1/2$ ($1/2 \leq p \leq 1$).
\end{lemma}
 To prove this lemma, let $u_i(p)$ 
 be the chance the walk goes up next when it is at level $i$ given it returns to 0 before reaching either  $+k$ or $-k$, for $i = -k+1,-k+2,...,k-1$. Note that, by symmetry, $u_i(p)=1-u_{-i}(1-p)$ is the chance the walk goes down next from level $-i$ given it returns to 0 before reaching $\pm k$, where now the probability of the walk going up is $1-p$. We show in Lemma \ref{lem2} (2) that $u_i(p)=u_i(1-p)$, which then implies that the probability of moving away from 0 given the walk returns to 0 before reaching $\pm k$ depends only on the player imbalance, $|p-(1-p)|=2|p-1/2|$, not on the direction of imbalance. Moreover, from Lemma \ref{lem2} (3), this probability is decreasing in the player imbalance.

\begin{lemma}  \label{lem2}
The following hold for $0\leq p\leq 1$:
\begin{enumerate}
\item $\IP(T_0<T^{|k|})$ is strictly increasing (decreasing) in $p$ when $p<1/2$  (when $p>1/2$).
\item $u_i(p)=u_i(1-p)$ for $-k+1 \leq i \leq k-1$. 
\item When $1\leq i\leq k-2$, $u_{i}(p)$ is strictly increasing (decreasing) in $p$ when $p< 1/2$ (when $p>1/2$).
\end{enumerate} 
 \end{lemma}

\begin{proof}[Proof of Lemma \ref{lem2}]  For Part 1, note that any path of length $2n$ that starts and ends at 0 without first visiting $k$ or $-k$ must contain $n$ upward steps and $n$ downward steps, and thus its probability equals $(p(1-p))^n$, which is strictly increasing in $p$ for $p< 1/2$ and strictly decreasing in $p$ for $p>1/2$.  The result for $\IP(T_0<T^{|k|})$ follows, since it is the sum of the probabilities of many such paths. This gives Part 1.

For Parts 2 and 3, we use the notation $\IP_i$ to denote the measure conditioned on starting the walk at level $i$.  From our observation about the symmetry of $u_i(p)$, pick $i>0$ without loss of generality, and let $T^j(p)$ be the time to hit level $j \geq 0$. We have

\begin{eqnarray*}
u_{i}(p)&=&\frac{p \IP_{i+1}(T_0<T^k)}{\IP_i(T_0<T^k)}=\frac{p \IP_{i+1}(T^i<T^k)\IP_i(T_0<T^k)}{\IP_i(T_0<T^k)}
\\
&=&p \IP_{i+1}(T^i<T^k),
\end{eqnarray*}
which we combine with
\[
\IP_{i+1}(T^i<T^k)=1-p+p\IP_{i+2}(T^{i+1}<T^k)\IP_{i+1}(T^{i}<T^k),
\]%
to get
\besn{\label{eq8}
u_i(p) = p(1-p)+u_{i+1}(p)u_{i}(p).
}
For Part 2 we use backward induction.  First note the boundary case $u_{k-1}(p)=u_{k-1}(1-p)=0$. We use the induction hypothesis  $u_{i+1}(p)=u_{i+1}(1-p)$ in the second equality below and (\ref{eq8}) in the first equality below to get
$$u_i(p)=\frac{p(1-p)}{1-u_{i+1}(p)} = \frac{p(1-p)}{1-u_{i+1}(1-p)} =u_i(1-p)$$ and thus Part 2 holds for all $i>0$ and then for all $i$ looking at the walk reflected across the origin.

For Part 3, we again use backward induction.  First note that $u_{k-1}(p)=0$ and \eq{eq8} with $i=k-2$ gives $u_{k-2}(p)= p(1-p)$, which is increasing in $p$ for $p< 1/2$ and decreasing for $p>1/2$.   Taking the derivative, from equation \eq{eq8} we get
$$u'_i(p) = 1-2p+u'_i(p)u_{i+1}(p) + u_i(p)u'_{i+1}(p) = \frac{1-2p+u_i(p)u'_{i+1}(p)}{1-u_{i+1}(p)}$$  which is positive when $p< 1/2$ and  negative when $p>1/2$ using the induction hypothesis that $u'_{i+1}(p)>0$ for $p< 1/2$ and $u'_{i+1}(p)<0$ for $p>1/2.$  Thus Part 3 holds.

\end{proof}

\begin{proof}[Proof of Lemma \ref{lem1}] 
%
Let $Y_{i}(p)$ be the time for the walk to go from level $i$
to 0 given it hits 0 before it hits $\pm k$, and for fixed $p$, $0<i<k$. Let 
$Z_{i}(p)$ be the time for the walk to go from level $i$ to $\pm k$ given it
hits $\pm k$ before it hits $0$, and for fixed $p$, $0<i<k$. Then, from the
structure of the random walk, and with $I(p)\sim $ Bernoulli$(p)$, 
\begin{eqnarray*}
Y(j) &=_{st}& I(p)Y_{1}(p)+(1-I(p))Y_{-1}(p) \\
Z(j) &=_{st}& I(p)Z_{1}(p)+(1-I(p))Z_{-1}(p) \\
Y_{i}(p) &=_{st}&Y_{-i}(1-p) \\
Z_{i}(p) &=_{st}&Y_{-(k-i)}(1-p).
\end{eqnarray*}%
The results for $Y(j)$ and $Z$ will then follow once we show that $%
Y_{i}(p)\leq _{st}Y_{i}(p^{\prime })$ for $0\leq p\leq p^{\prime }\leq 1/2$
and $0<i<k$, and that $Y_{1}(p)=_{st}Y_{-1}(p)$ and $Z_{1}(p)=_{st}Z_{-1}(p)$
for all $p$. That $Y_{1}(p)=_{st}Y_{-1}(p)$ and $Z_{1}(p)=_{st}Z_{-1}(p)$
follows from Part 2 of Lemma \ref{lem2} and the last two equations above. We will show that $Y_{i}(p)\leq _{st}Y_{i}(p^{\prime })$ by coupling the walks starting in state $i$ so that 
$Y_{i}(p)\leq Y_{i}(p^{\prime })$ with probability 1.
Whenever the walks are at 
different levels, we let them step independently; since the walks both started at the same level, they will remain an even number of steps apart and thus after this step can meet but not cross.  When the walks are at 
the same level, by Part 3 of Lemma \ref{lem2} the chance of a step up is at least as large for the second walk as it is for the first 
walk, so we can couple the steps so the second walk 
is always no 
closer to the origin than the first walk.  This means 
we will have $Y_i\leq Y_i(p')$ and thus $Y_i\leq_{st} Y_i(p').$
The result for $N$ follows from Part 1 of Lemma \ref{lem2}.
\end{proof}

\begin{proof}[Proof of Theorem \ref{thm1}]
We can similarly couple $T^{|k|}(p')$ together with $T^{|k|}(p)$ on the same probability space (making the dependency on $p$ explicit in the notation), using our decomposition and Lemma \ref{lem2}, so that $T^{|k|}(p')\geq T^{|k|}$ almost surely and the Theorem follows.
\end{proof}

\section{An inductive argument}

We will use induction on $k$ in our second proof of Theorem 1. Consider the
following, alternative, decomposition, in which the time to play a game when the boundary
is $\pm k$ (call this a size $k$ game) is a random sum of times to play
smaller games. In particular, we first play a size 1 game, at which time the
random walk is at $\pm 1$ with respective probabilities $p$ and $1-p$. Then
we play a size $k-1$ game, at the end of which the position of the random
walk will be in the set $\{-k,-(k-2),k-2,k\}$, depending on whether it was
at $+1$ or $-1$ at the end of the first game, and depending on whether it
went up or down by $k-1$ steps from there. If the position at the end of the
second game is $\pm (k-2)$, then we play a third game of size $2$, and so
on, until we finally reach $\pm k$ and stop. Let $N$ be the total number of such
games, and let $r(n)=\IP\{N=n|N\geq n\}$ be its hazard rate. Let $y(i)$ be
such that at the end of the $i$'th game, the possible positions of the
random walk are $\pm k$ and $\pm y(i)\neq \pm k$
and let $d(i)$ be the size of the $i$'th game. That is, let 
$y(1)=d(1)=1$, and, for $i=2,3...$, if $%
y(i-1)>0$ let $d(i)=k-y(i-1)$ and $y(i)=|y(i-1)-d(i)|$, and if $y(i-1)=0$
let $d(i)=y(i)=1$. Note that for fixed $k$, $y(i)$ and $d(i)$ are
deterministic sequences. 
We therefore have, by construction and from the
Markov property of our random walk, the following.

\begin{lemma}
	\label{decomp}For $k>1$,%
	\[
	T^{|k|}=\sum_{i=1}^{N}T^{|d(i)|} 
	\]%
	where the $T^{|d(i)|}$'s are independent.
\end{lemma}

We note that $y(i)$, $d(i)$, $r(n)$, (and $N$) all depend on $k$, but we suppress this dependency in the notation.

Let $\pi _{k}^{+}(p)=\IP (T^k<T^{-k})$ be the probability that the walk $S=\{S_{n}\}_{n=0}^{%
	\infty }$, with $S_{0}=0$, ends at $+k$ before hitting $-k$, so $\pi
_{1}^{+}(p)=p$ and $\pi _{k}^{+}(p)=p^{k}/(p^{k}+(1-p)^{k})$, from the
classic Gambler's Ruin formula (\cite{Feller1968}). Let $\tau
(j)=\sum_{i=1}^{j}T^{|d(i)|}|N>j$ be the time to play $j$ subgames,
conditioned on $S$ having not hit $\pm k$ in the first $j$ games (so $%
S_{\tau (j)}=\pm y(j)$).

\begin{lemma}
	\label{indep}For any $j$, $\tau (j)$ is independent of the event $\{S_{\tau
		(j)}=+y(j)\}$, and 
	\[
	\IP(S_{\tau (j)}=+y(j))=\frac{p^{y(j)}}{p^{y(j)}+(1-p)^{y(j)}}=\pi
	_{y(j)}^{+}(p). 
	\]
\end{lemma}

\begin{proof}
	Fix $j$, Condition on $N>j$ and $\tau (j)=t$, and for ease of notation, let $y=y(j)$.
	Let $\rho ^{+}$ be an arbitrary path of $S$ that ends up at $+y$ at time $t$
	without hitting $\pm k$ before time $t$. Then, on path $\rho ^{+}$, $S$ has
	a total of $(t-y)/2+y$ upward steps and $(t-y)/2$ downward steps, so the
	probability of path $\rho ^{+}$ is $p^{(t-y)/2}(1-p)^{(t-y)/2}p^{y}$. Let $%
	\rho ^{-}$ be the same path reflected through the origin (interchanging up
	and down steps), so the path $\rho ^{-}$ ends up at $-y$ at time $t$, also
	without hitting $\pm k$ before time $t$. The probability of path $\rho ^{-}$
	is $p^{(t-y)/2}(1-p)^{(t-y)/2}(1-p)^{y}$. Let $A(\rho ^{\pm })$ be the event
	that $S$ takes path $\rho ^{+}$ or $\rho ^{-}$ before hitting $\pm k$. Then 
	\begin{eqnarray*}
		\IP(S_{\tau (j)} &=&+y,A(\rho ^{\pm })|\tau (j)=t) \\
		&=&\IP(S_{\tau (j)}=+y|A(\rho ^{\pm }),\tau (j)=t)\IP(A(\rho ^{\pm })|\tau (j)=t)
		\\
		&=&\frac{p^{(t-y)/2}(1-p)^{(t-y)/2}p^{y}}{%
			p^{(t-y)/2}(1-p)^{(t-y)/2}p^{y}+p^{(t-y)/2}(1-p)^{(t-y)/2}(1-p)^{y}}\IP(A(\rho
		^{\pm })|\tau (j)=t) \\
		&=&\frac{p^{y}}{p^{y}+(1-p)^{y}}\IP(A(\rho ^{\pm })|\tau (j)=t).
	\end{eqnarray*}%
	Summing over all such pairs of reflected paths, $\rho ^{+}$ and $\rho ^{-}$,
	of length $t$, we have for each $t$, 
	\[
	\IP(S_{\tau (j)}=+y|\tau (j)=t)=\frac{p^{y}}{p^{y}+(1-p)^{y}}=\pi _{y}^{+}(p)%
	\text{.} 
	\]
	Thus, the event that $S_{\tau(j)}=+y$ is independent of $\tau(j)$, and $\IP(S_{\tau (j)}=+y)=\pi _{y}^{+}(p)$.
\end{proof}

A similar argument gives the following corollary.

\begin{corollary}
	\label{Tk indep}The random variable $T^{|k|}(p)$ is independent of whether the
	walk first hits $+k$ or $-k$ for all $p$ and for any $k$.
\end{corollary}

The following are also corollaries of Lemma \ref{indep}.

\begin{corollary}
	\label{pi increasing}For $0\leq p\leq 1/2$, $\pi _{k}^{+}(p)$ is increasing
	in $p$, from $\pi _{k}^{+}(0)=0$ to $\pi _{k}^{+}(1/2)=1/2$.
\end{corollary}

\begin{corollary}
	In the decomposition of Lemma \ref{decomp}, for fixed $k>1$, $N$ is independent
	of the $T^{|d(i)|}$'s, and, for $n=1,2,...$, 
	\[
	r(n)=\left\{ 
	\begin{array}{cc}
		\pi _{y(n-1)}^{+}(p)\pi _{d(n)}^{+}(p)+(1-\pi _{y(n-1)}^{+}(p))(1-\pi _{d(n)}^{+}(p)) & 
		\text{if }y(n-1)>0 \\ 
		0 & \text{if }y(n-1)=0%
	\end{array}%
	\right\} \text{.}
	\]
\end{corollary}

\begin{proof}
	That $N$ is independent of the $T^{|d(i)|}$'s follows from Corollary \ref{Tk indep}, because, for each subgame, the probability that it is the last game is independent of the time to play the game. That $r(n)=0$ if $y(n-1)=0$ is immediate because then $y(n)=d(n)=1$. From Lemma \ref{indep}, for $y(n-1)>0$ and given $N\geq n$,
	the probability that the walk is at $y(n-1)$ (respectively, $-y(n-1)$) at the end of the $%
	n-1$'st game, i.e., at time $\tau (n-1)$, is $\pi _{y(n-1)}^{+}(p)$ (respectively, $1-\pi
	_{y(n-1)}^{+}(p)$). The probability that the walk hits $\pm k$ at the end of
	the $n$'th game given it has not yet hit $\pm k$, is the probability of
	going up by $d(n)$ from $y(n-1)$ or going down by $d(n)$ from $-y(n-1)$,
	i.e., 
	$$
	r(n)=\IP\{N=n|N\geq n\}=\pi _{y(n-1)}^{+}(p)\pi _{d(n)}^{+}(p)+(1-\pi
	_{y(n-1)}^{+}(p))(1-\pi _{d(n)}^{+}(p)).
	$$
\end{proof}

\begin{proof}[Proof of Theorem \ref{thm1}]
	We use induction on $k$, where the result follows trivially for $%
	T^{|1|}\equiv 1$. Fix $k$ and assume the result holds for all $%
	j=1,...,k-1$, and therefore for each subgame $i$ because $d(i)<k$. That is, assume $T^{|d(i)|}$ is stochastically increasing in $p$ (for $0 \leq p \leq 1/2$). Then we will have the result for $k$, i.e., $T^{|k|}$ is stochastically increasing in $p$, from Lemma \ref{decomp} and
	the induction hypothesis, as long as $N$ is stochastically increasing in $p$.
	We will show the stronger result, that $N$ gets larger in the hazard rate
	sense as $p$ increases, i.e., its hazard rate $r(n)$ is decreasing
	(nonstrictly) in $p$ for all $n$. If $y(n-1)=0$ then $r(n)=0$ is trivially
	decreasing in $p$. Define, for fixed $n$ such that $y(n-1)\neq 0$, 
	\[
	g_{n}(p)=1-r(n)=\pi _{y(n-1)}^{+}(p)(1-\pi _{d(n)}^{+}(p))+(1-\pi
	_{y(n-1)}^{+}(p))\pi _{d(n)}^{+}(p).
	\]%
	Then, taking the derivative with respect to $p$ and using Corollary \ref{pi
		increasing} for $0 \leq p \leq 1/2$, 
	\begin{eqnarray*}
		g_{n}^{\prime }(p) &=&\pi _{y(n-1)}^{+\prime }(p)+\pi _{d(n)}^{+\prime
		}(p)-2\pi _{y(n-1)}^{+\prime }(p)\pi _{d(n)}^{+}(p)-2\pi _{y(n-1)}^{+}(p)\pi
		_{d(n)}^{+\prime }(p) \\
		&\geq &\pi _{y(n-1)}^{+\prime }(p)+\pi _{d(n)}^{+\prime }(p)-\pi
		_{y(n-1)}^{+\prime }(p)-\pi _{d(n)}^{+\prime }(p)=0\text{.}
	\end{eqnarray*}
\end{proof}

\begin{remark}
	We note that the decomposition can be simplified when $k$ is even by choosing $%
	y(1)=k/2$ instead of $y(1)=1$:%
	\[
	T^{|k|}(p)=\sum_{i=1}^{N}\left( T^{|k/2|}_i+ \tilde{T}^{|k/2|}_i\right) 
	\]%
	where $T^{|k/2|}_i =_{st} \tilde{T}^{|k/2|}_i =_{st} T^{|k/2|}$, and $N$, $T^{|k/2|}_i$, and 
	$\tilde{T}^{|k/2|}_i$ for all $i$ are independent, and 
	\[
	N\sim \text{ Geometric}((\pi _{k/2}^{+}(p))^{2}+(1-\pi _{k/2}^{+}(p))^{2}).
	\]
\end{remark}

\section{A corollary for Brownian motion with drift}

Letting $B(t)$ denote standard Brownian motion at time $t$, for $k>0$ consider the time 
$$T_{\mu}=\inf\{t\geq 0:\mu t+B(t) \notin [-k,k]\}$$ 
that Brownian motion with drift $\mu$ first exits the interval $[-k,k].$
\cite{Bor2002}[p. 309 equation (3.0.2), and p. 641] gives the probability density function for $T_{\mu}$ as
\besn{\label{eq4}\IP(T_{\mu} \in dt) = ( e^{-\mu k}  + e^{\mu k})  e^{-\mu^2t/2} \sum_{i=-\infty}^\infty \frac{k+4ik}{\sqrt{2\pi} t^{3/2}} e^{-(k+4ik)^2/(2t)}dt.}
Using weak convergence of the paths of asymmetric simple random walk to Brownian motion with drift, we obtain the following immediate corollary from Theorem \ref{thm1}.  
\begin{corollary}\label{c1}
With the definitions above, if $0\leq \mu \leq \mu^{\prime}$ then $T_{\mu} \geq_{st}  T_{\mu^{\prime}}$.
\end{corollary}
This result shows that the exit time from a symmetric interval for Brownian motion with 
drift is stochastically increased as the drift is moved towards zero.  In particular, it is stochastically maximized when there is no drift.  As far as we can tell, this result is new and does not seem easily obtainable from the explicit form (\ref{eq4}) of the density.

%
%

\section{ACKNOWLEDGEMENTS}
We greatly appreciate the feedback from the referees, which improved the presentation.

\end{document}